\documentclass{amsart}
\usepackage{amsmath, amssymb, amsthm, amscd, amsfonts, eucal, hyperref}
\usepackage[all]{xy}
\usepackage{enumerate}
\usepackage[T5, T1]{fontenc}
\usepackage{color, charter}

\newtheorem{theorem}{Theorem}[section]
\newtheorem{proposition}[theorem]{Proposition}
\newtheorem{lemma}[theorem]{Lemma}

\newtheorem {corollary}[theorem]{Corollary}
\theoremstyle {definition}
\newtheorem {definition}[theorem]{Definition}
\newtheorem {example}[theorem]{Example}

\theoremstyle {remark}
\newtheorem{remark}[theorem]{Remark}

\def\Spec{\operatorname{Spec}}

\def\Ass{\operatorname{Ass}}

\def\Supp{\operatorname{Supp}}
\def\Psupp{\operatorname{Psupp}}

\def\Ann{\operatorname{Ann}}

\def\Ext{\operatorname{Ext}}

\def\Ker{\operatorname{Ker}}
\def\Coker{\operatorname{Coker}}

\newcommand{\fm}{\ensuremath{\mathfrak m}}

\newcommand{\fp}{\ensuremath{\mathfrak p}}

\newcommand{\bZ}{\ensuremath{\mathbb Z}}

\begin{document}

\title[On the length function of saturations of ideal powers]{On the length function of saturations of ideal powers}

\author{\fontencoding{T5}\selectfont \DJ o\`an Trung C\uhorn{}\`\ohorn ng} \address{\fontencoding{T5}\selectfont \DJ o\`an Trung C\uhorn{}\`\ohorn ng, Institute of Mathematics and the Graduate University of Science and Technology, Vietnam Academy of Science and Technology, 18 Hoang Quoc Viet, 10307 Hanoi, Viet Nam.} \email{dtcuong@math.ac.vn}

\author{\fontencoding{T5}\selectfont Ph\d am H\`\ocircumflex ng Nam} \address{\fontencoding{T5}\selectfont Ph\d am H\`\ocircumflex ng Nam, 
University of Sciences, Thai Nguyen University, Thai Nguyen, Vietnam.} \email{phamhongnam2106@gmail.com}

\author{\fontencoding{T5}\selectfont Ph\d am H\`ung Qu\'y} \address{\fontencoding{T5}\selectfont Ph\d am H\`ung Qu\'y, Department of Mathematics, FPT University, and Thang Long Institute of Mathematics and Applied sciences, Thang Long University, Hanoi, Vietnam.} \email{quyph@fe.edu.vn}



\begin{abstract}
For an ideal $I$ in a local ring $(R, \fm)$, we prove that the integer-valued function $\ell_R(H^0_\fm(R/I^{n+1}))$ is a polynomial for $n$ big enough if either $I$ is a principle ideal or $I$ is generated by part of an almost p-standard system of parameters. Furthermore, we are able to compute the coefficients of this polynomial in terms of length of certain local cohomology modules and usual multiplicity if either the ideal is principal or it is generated by part of a standard system of parameters in a generalized Cohen-Macaulay ring. We also give an example of an ideal generated by part of a system of parameters such that the function $\ell_R(H^0_\fm(R/I^{n+1}))$ is not a polynomial for $n\gg 0$.
\end{abstract}

\maketitle

\section{Introduction}

Let $(R, \fm)$ be a Noetherian local ring and $I\subset R$ be an ideal. There is a numerical function attached to $I$,
$$h^0_I: \bZ_{\geq 0}\rightarrow \bZ_{\geq 0}, n\mapsto \ell_R(H^0_\fm(R/I^{n+1})).$$
If $I$ is $\fm$-primary then $h^0_I(n)=\ell(R/I^{n+1})$ is the Hilbert-Samuel function. So it is of polynomial type, that means, there is a polynomial $H_I(n)$ such that $h^0_I(n)=H_I(n)$ for all $n\gg 0$. For a general ideal $I$, it is natural to ask whether the function $h^0_I(n)$ is of polynomial type. Unfortunately, it is not always the case. In \cite{UV} Ulrich and Validashti proved that the limit
$$\limsup_n d!\frac{h^0_I(n)}{n^d},$$
is a finite number. They showed that this limit can be considered as a generalization of the Buchsbaum-Rim multiplicity and called it the $\epsilon$-multiplicity. This numerical invariant $\epsilon$-multiplicity is denoted by $\epsilon(I)$ and is very useful in the theory of equisingularity (see \cite{KUV}). Nevertheless, Cutkosky, Ha, Srinivasan and Theodorescu \cite[Theorem 2.2]{CTST} gave an interesting example of a $4$-dimensional local ring $R$ and an ideal $I$ such that
$$\epsilon(I)=\lim_{n\to \infty}4!\frac{h^0_I(n)}{n^4},$$
is an irrational number. Thus in this example $h^0_I(n)$ is not of polynomial type. So the answer to the previous question is no in general.

In this short note, we address the question under which assumption the function $h^0_I(n)$ is of polynomial type, that is, there is a polynomial $P_I(n)$ such that $h^0_I(n)=P_I(n)$ for all $n\gg 0$. Note that $H^0_\fm(R/I^n)=\cup_{s>0}(I^n:_R\fm^s)/I^n$. As the main results, we will show that $h^0_I(n)$ is of polynomial type in the following cases:

\begin{enumerate}
\item[(a)] The ideal $I$ is principal (Theorem \ref{22});
\item[(b)] $R$ is unmixed and the ideal $I$ is generated by part of an almost p-standard system of parameters of $R$ (Theorem \ref{37}). 
\end{enumerate}

For the case (a), when $I$ is a principle ideal, the main idea of the proof is to relate the function $h^0_I(n)$ to the Hilbert functions of some Artinian modules arising from the first local cohomology module of $R$. In this case the leading coefficient of the corresponding polynomial is expressed precisely in terms of usual multiplicity and length of certain local cohomology modules.

In the case (b), we consider those ideals generated by part of an almost p-standard system of parameters. The later notion was introduced by the first two authors in \cite{DTCNam}. We will show that $h^0_I(n)=\ell_R(J^{n+1}/I^{n+1})$ for an ideal $J$ such that $I$ is a reduction of $J$ and $\ell_R(J/I)$ is finite. Then a theorem of Amao \cite{AMAO} concludes that there is a polynomial $P_I(n)$ such that $h^0_I(n)=P_I(n)$ for all $n\gg 0$. It should be remarked that $h^0_I(n)$ and $P_I(n)$ are also called the Rees function and Rees polynomial of the pair $(I, J)$ by Herzog-Puthenpurakal-Verma \cite{HPV}. If $I$ is generated by part of a standard system of parameters, we are able to give a precise formula for $P_I(n)$ (Theorem \ref{39}).

Relying on the example of Cutkosky, Ha, Srinivasan and Theodorescu \cite[Theorem 2.2]{CTST}, we give an example of an ideal $I$ generated by part of a system of parameters such that the function $h^0_I(n)$ is not of polynomial type (Example \ref{310}).

The article consists of three sections. We treat the case of principal ideals in Section 2. Ideals generated by part of a system of parameters are considered in Section 3.

Throughout this note, $(R, \fm)$ is always a Noetherian local ring.


\section{Length function of saturations of powers of a principal ideal}

In this section we study the existence of a saturated Hilbert polynomial of the ring $R$ with respect to a principal ideal. Before stating the main result of this section, we first recall several facts about Hilbert functions and multiplicities of Artinian modules from \cite{Kir, NTCNhan} which will be used in the subsequent section.

\begin{remark}\label{21} Let $A$ be an Artinian $R$-module and $I\subset R$ be a proper ideal. Suppose that $0:_AI$ is of finite length.
\begin{enumerate}
\item[(a)] Kirby \cite[Proposition 2]{Kir} proved that the module $0:_AI^{n+1}$ is also of finite length for any $n\geq 0$ and moreover, there is a polynomial $P_{A, I}(n)$ such that $P_{A, I}(n)=\ell_R(0:_AI^{n+1})$ for all $n\gg 0$. The degree of $P_{A, I}(n)$ is bounded above by $\mu(I)$.
\item[(b)] (See \cite{NTCNhan}) The multiplicity $e^\prime(I; A)$ of the Artinian module $A$ with respect to $I$ is defined such that $d!e^\prime(I;A)$ is the leading coefficient of the polynomial $P_{A, I}(n)$, here $d$ is the degree of $P_{A, I}(n)$. Given a short exact sequence of Artinian modules
$$0\rightarrow A_1\rightarrow A\rightarrow A_2\rightarrow 0,$$
we have $e^\prime(I, A)=\delta_{dd_1}e^\prime(I, A_1)+\delta_{dd_2}e^\prime(I, A_2)$, where $d_i$ is the degree of the polynomial $P_{I, A_i}(n)$ and $\delta_{dd_i}$ is the Kronecker delta, $i=1, 2$.
\item[(c)] Assume that $R$ is a homomorphic image of a Cohen-Macaulay local ring. Equivalently, as shown in \cite{Kaw} (see also \cite{NTCDTC2}), $R$ is universally catenary and all its formal fibers are Cohen-Macaulay. In \cite{BS1} Brodmann and Sharp defined the $i$-th pseudo-support of $R$ to be the set
$$\Psupp^i(R)=\{\fp\in \Spec(R): H^{i-\dim R/\fp}_{\fp R_\fp}(R_\fp)\not=0\}.$$
In particular, $\Psupp^1(R)=\{\fm\}\cup\{\fp:
\dim R/\fp=1 \text{ and }H^0_{\fp R_\fp}(R_\fp)\not=0\}$. Put $\mathrm{psd}^i(R)=\max\{\dim R/\fp: \fp\in \Psupp^i(R)\}$. They then showed that (see \cite[Theorem 2.4]{BS1})
$$e^\prime(I; H^i_\fm(R))=\sum_{\substack{\fp\in\Psupp^i(R)\\ \dim R/\fp=\mathrm{psd}^i(R)}}\ell_{R_\fp}(H^{i-\dim R/\fp}_{\fp R_\fp}(R_\fp))e(I, R/\fp).$$
\end{enumerate}
\end{remark}

The main result of this section is the following theorem

\begin{theorem}\label{22}
Let $I=aR$ be a principal ideal of $R$. There is a polynomial $P_I(n)$ with $\deg P_I(n)\leq 1$ such that $P_I(n)=h^0_I(n)$ for all $n\gg 0$.
\end{theorem}
\begin{proof}
As $R$ is Noetherian, there is a number $t>0$ such that $0:_Ra^t=0:_Ra^{n+t+1}$ for any $n\geq 0$. The short exact sequence
$$0\longrightarrow R/0:_Ra^t\xrightarrow{\ *a^{n+t+1}\ } R\longrightarrow R/a^{n+t+1}R\longrightarrow 0,$$
leads to an exact sequence of local cohomology modules
\begin{multline*}
0\longrightarrow H^0_\fm(R/0:_Ra^t)\xrightarrow{\ *a^{n+t+1}\ } H^0_\fm(R)\longrightarrow H^0_\fm(R/a^{n+t+1}R)\\
\longrightarrow H^1_\fm(R/0:_Ra^t)\xrightarrow{\psi_{n+t+1}} H^1_\fm(R)\longrightarrow \ldots.
\end{multline*}
Here $H^0_\fm(R/0:_Ra^t)=0$ as $a$ is a regular element on $R/0:_Ra^t$. The map $\psi_{n+t+1}$ is derived from the multiplication by $a^{n+t+1}$ on $R$. We  get
$$h^0_I(n+t)=\ell(\Ker \psi_{n+t+1})+\ell_R(H^0_\fm(R)).$$
From the commutative triangle\medskip

\centerline{
\xymatrix@1{
R/0:_Ra^t \ \ \ar[rr]^{*a^{n+t+1}} \ar[dr]_{*a^{n+1}} && \ R \\
&R/0:_Ra^t\ar[ur]_{*a^t}
}}
\bigskip

\noindent there induces a commutative triangle of local cohomology modules \medskip

\centerline{
\xymatrix@1{
H^1_\fm(R/0:_Ra^t) \ \ \ar[rr]^{\psi_{n+t+1}} \ar[dr]_{*a^{n+1}} && \ H^1_\fm(R). \\
&H^1_\fm(R/0:_Ra^t)\ar[ur]_{\psi_t}
}}
\bigskip

\noindent Hence $\Ker(\psi_{n+t+1})=\Ker(\psi_t):_Aa^{n+1}$, where $A=H^1_\fm(R/0:_Ra^t)$ is an Artinian module (see, for example, \cite[Theorem 7.1.3]{BS}). Let $\bar A=A/\Ker(\psi_t)$. We obtain
$$h^0_I(n+t)=\ell(0:_{\bar A}a^{n+1})+\ell(\Ker \psi_t)+\ell_R(H^0_\fm(R)),$$
for any $n\geq 0$. Since the module $\bar A$ is Artinian, the conclusion is implied from Remark \ref{21}(a).
\end{proof}

Let $I$ be a principal ideal and let $P_I(n)$ be the polynomial in Theorem \ref{22} so that $P_I(n)=h^0_I(n)$ for all $n\gg 0$. We have seen in the proof of Theorem \ref{22}, the polynomial $P_I(n)$ is closely related to a Hilbert polynomial of a first local cohomology module, thus, an Artinian module. Since $\deg P_I(n)\leq 1$, if the dimension of $R$ is bigger than $1$ then
$$\lim_nd!\frac{P_I(n)}{n^d}=0.$$
So the $\epsilon$-multiplicity  of $R$ with respect to a principal ideal is zero.

In the next, we will investigate more on the coefficients of the polynomials $P_I(n)$. We write 
$$P_I(n)=ne^{sat}_0(a; R)+e_1^{sat}(a; R),$$
where $I=aR$ and the coefficients $e^{sat}_0(a; R), e^{sat}_1(a; R)$ are integers. We will show that the leading coefficient $e_0^{sat}(a; R)$ could be expressed explicitly by means of usual multiplicities and length of certain local cohomology modules. We need first a lemma.

\begin{lemma}\label{23}
Denote $\Ass(R)_1=\{\fp\in\Ass(R): \dim R/\fp=1\}$. Then
$$\Ass(R)_1=\Psupp^1(R)\setminus\{\fm\}.$$
\end{lemma}
\begin{proof}
Let $\fp$ be a prime ideal of dimension $1$. Then $H^0_{\fp R_\fp}(R_\fp)\not=0$ if and only if $\fp R_\fp$ is an associated prime ideal of $R_\fp$. This turns out to be equivalent to that $\fp$ is an associated prime ideal of $R$.
\end{proof}

From now on, we will need to assume that $R$ is a quotient of a Cohen-Macaulay local ring.

\begin{theorem}\label{24}
Let $R$ be a quotient of a Cohen-Macaulay Noetherian local ring. Let $I=aR$ be a principal ideal of $R$. We have
$$e^{sat}_0(a; R)=\sum_{\fp\in\Ass(R)_1\setminus V(I)}\ell_{R_\fp}(H^0_{\fp R_\fp}(R_\fp))e(a; R/\fp).$$
\end{theorem}
\begin{proof}
Using the notations in the proof of Theorem \ref{22}, from the equality
$$h^0_I(n+t)=\ell(0:_{\bar A}a^{n+1})+\ell(\Ker \psi_t)+\ell_R(H^0_\fm(R)),$$
where $A=H^1_\fm(R/0:_Ra^t)$ and $\bar A=A/\Ker(\psi_t)$, we get that $$e^{sat}_0(a; R)=e^\prime (a; \bar A)=e^\prime(a; H^1_\fm(R/0:_Ra^t)/\Ker(\psi_t)).$$
Since the module $\Ker(\psi_t)$ is of finite length as it is the image of the connecting map $H^0_\fm(R/a^tR)\rightarrow H^1_\fm(R/0:_Ra^t)$, Remark \ref{21}(b) gives us
$$e^{sat}_0(a; R)=e^\prime (a; H_\fm ^1(R/0:_Ra^t)).$$
Using Brodmann-Sharp's associativity formula for multiplicities of local cohomology modules in Remark \ref{21}(c) we have
\begin{displaymath}
e^\prime (a, H_\fm ^1(R/0:_Ra^t))=
\sum_{\substack{\fp \in \Psupp^1(R/0:_Ra^t) \\ \fp \not=\fm}}
\ell_{R_\fp }(H_{\fp R_\fp }^0(R_\fp/(0:_Ra^t)_{\fp }))e(a; R/\fp ).
\end{displaymath}

Now the inclusion $R/0:_Ra^t\stackrel{*a^t}{\longrightarrow} R$ leads to an inclusion
$\Psupp^1(R/0:_Ra^t)\subseteq \Psupp^1(R)$. Take a prime ideal $\fp \in \Psupp^1(R)$ with $\dim R/\fp =1$. If $a\not\in\fp$ then $(0:_Ra^t)_\fp=0$ and $R_\fp\simeq (R/0:_Ra^t)_\fp$. So $H_{\fp R_\fp }^0(R_\fp)\simeq H_{\fp R_\fp }^0((R/0:_Ra^t)_\fp)$. This implies that
$$\Psupp^1(R)\setminus V(I)=\Psupp^1(R/0:_Ra^t)\setminus V(I).$$
On the other hand, if $a\in \fp$ then $a$ is a regular element on $R/0:_Ra^t$, hence $\frac{a}{1}$ is a regular element on $(R/0:_Ra^t)_\fp$. This shows particularly that $H_{\fp R_\fp }^0((R/0:_Ra^t)_\fp)=0$ and thus
$$\Psupp^1(R/0:_Ra^t)\setminus V(I)=\Psupp^1(R/0:_Ra^t)\setminus \{\fm\}.$$
Therefore we get

\begin{displaymath}
\begin{split}
e^{sat}_0(a; R)
&= \sum_{\substack{\fp \in \Psupp^1(R/0:_Ra^t) \\ \fp \not=\fm}}
\ell_{R_\fp }(H_{\fp R_\fp }^0(R_\fp/(0:_Ra^t)_{\fp }))e(a; R/\fp )\\
&= \sum_{\substack{\fp \in \Psupp^1(R/0:_Ra^t) \\ \fp \not\in V(I)}}
\ell_{R_\fp }(H_{\fp R_\fp }^0(R_\fp/(0:_Ra^t)_{\fp }))e(a; R/\fp )\\
&=\sum_{\fp \in \Psupp^1(R)\setminus V(I)}
\ell_{R_\fp }(H_{\fp R_\fp }^0(R_{\fp }))e(a; R/\fp)\\
&=\sum_{\fp \in \Ass(R)_1\setminus V(I)}
\ell_{R_\fp }(H_{\fp R_\fp }^0(R_{\fp }))e(a; R/\fp),
\end{split}
\end{displaymath}
as required, here the last equality follows from Lemma \ref{23}.
\end{proof}

A direct consequence of Theorem \ref{24} is that the saturated Hilbert polynomial $P_I(n)$ is of degree one if and only if $\Ass(R)_1\setminus V(I)\not=\emptyset$, or equivalently, $a$ is in some associated prime ideal of dimension one. In the next, for a given ring $R$ and let $a$ vary, we consider the two extremal cases: $\Ass(R)_1\setminus V(I)=\emptyset$ and $\Ass(R)_1\setminus V(I)=\Ass(R)_1$.

\begin{corollary}\label{25}
Let $a\in \fm$ and let $t>0$ such that $0:_Ra^{n+t+1}=0:_Ra^t$ for all $n\geq 0$.
\begin{enumerate}[(i)]
\item If $a$ is a filter-regular element of $R$ then
$$e_0^{sat}(a; R)=e^\prime(a; H^1_\fm(R)).$$
The right hand side is the multiplicity of a local cohomology module defined in Remark \ref{21}(b).

\item The element $a$ is in the radical of $\Ann_RH^1(R)$ if and only if $e_0^{sat}(a; R)=0$. In that case,
$$P_I(n)=\ell(H^1_\fm(R/0:_Ra^t))+\ell(H^0_\fm(R)),$$
is a constant polynomial.
\end{enumerate}
\end{corollary}
\begin{proof}
\item(i) Since $a$ is a filter-regular element of $R$, it is obvious that $0:_{H^1_\fm(R)}a$ is of finite length, thus the multiplicity $e^\prime(a; H^1_\fm(R))$ is determined by Remark \ref{21}(b).

On the other hand, being a filter-regular element of $R$ means that $a\not\in\fp$ for any associated prime ideal $\fp\not=\fm$. So $\Ass(R)_1\setminus V(I)=\Ass(R)_1$. From Theorem \ref{24} and the associativity formula in Remark \ref{21}(c) we obtain the equality $e_0^{sat}(a; R)=e^\prime(a; H^1_\fm(R))$.

\item(ii) By Theorem \ref{24}, the leading coefficient $e_0^{sat}(a; R)$ vanishes if and only if $\Ass(R)_1\subseteq V(I)$. For the necessary condition, we take an associated prime ideal $\fp\in \Ass(R)_1$. There is an inclusion $\psi: R/\fp\hookrightarrow R$ which give an exact sequence
$$H^0_\fm(\Coker \psi)\rightarrow H^1_\fm(R/\fp)\rightarrow H^1_\fm(R).$$
Since $a$ is in the radical of $\Ann_RH^1_\fm(R)$, the exact sequence implies that $a^sH^1(R/\fp)=0$ for some $s\gg0$. This implies in particular that $a\in \fp$ since $\dim R/\fp=1$. Hence $\Ass(R)_1\subseteq V(I)$.

Conversely, suppose $\Ass(R)_1\subseteq V(I)$. Since $R$ is a homomorphic image of a Cohen-Macaulay local ring, by \cite[Proposition 2.5]{BS1}, we have $$V(\Ann_RH^1_\fm(R))=\bigcup_{\substack{\fp\in \Ass(R)\\ \dim R/\fp\leq 1}} V(\fp)\subseteq V(I).$$
So $I\subseteq \sqrt{\Ann_RH^1_\fm(R)}$.
\end{proof}

If  $R$ is a homomorphic image of a Gorenstein local ring $(R^\prime, \fm^\prime)$, then the coefficient $e_0^{sat}(a; R)$ could be computed as a multiplicity of the first module of deficiency $K^1(R)=\Ext^{d^\prime-1}_{R^\prime}(R, R^\prime)$, where $d^\prime=\dim R^\prime$. The module of deficiency $K^1(R)$ is a finitely generated $R$-module which is dual to $H^1_\fm(R)$ via the Local Duality Theorem \cite[11.2.6]{BS}.

\begin{corollary}\label{26}
Let $(R, \fm)$ be a homomorphic image of a Gorenstein local ring. Let $I=aR$ be a principal ideal. We have
$$e_0^{sat}(a; R)=e(a; K^1(R)),$$
where the multiplicity on the right hand side is defined by
$$e(a; K^1(R))=\sum_{\substack{\fp\in\Ass K^1(R)\setminus V(I)\\ \dim R/\fp=1}}\ell_{R_\fp }(K^1(R)_\fp)e(a; R/\fp).$$
\end{corollary}
\begin{proof}
Following \cite[Proposition 1.2]{BS1}, we have $\Psupp^1(M)=\Supp K^1(R)$ being a closed subset of $\Spec R$ of dimension at most $1$. Moreover, for a prime ideal $\fp\in \Psupp^1(R)_1=\Supp K^1(R)_1=\Ass K^1(R)_1$, we have
$\ell_{R_\fp}(H^0_{\fp R_\fp}(R_\fp))=\ell_{R_\fp}(K^1(R)_\fp)$. The conclusion is therefore a consequence of Theorem \ref{24}.
\end{proof}


\section{Ideals generated by part of a system of parameters}

In this section we consider the ideals $I$ generated by part of a system of parameters and study whether the function $h^0_I(n)$ is of polynomial type. In the first part of the section, we restrict ourself to the case of almost p-standard systems of parameters which are defined in \cite{DTCNam}.

We recall the definition of almost p-standard system of parameters.

\begin{definition}\cite[Definition 2.1]{DTCNam}\label{31}
Let $M$ be a finitely generated $R$-module of dimension $d$. A system of parameters $x_1, \ldots, x_d$ of $M$ is a called an almost p-standard system of parameters if there are given integers $\lambda_0, \ldots, \lambda_d$ such that
$$\ell_R(M/(x_1^{n_1}, \ldots, x_d^{n_d})M)=\sum_{i=0}^dn_1\ldots n_i\lambda_i,$$
for all $n_1, \ldots, n_d>0$.
\end{definition}

Recall that a standard system of parameters satisfies
$$\ell_R(M/(x_1^{n_1}, \ldots, x_d^{n_d})M)=n_1\ldots n_de(x_1, \ldots, x_d;M)+I(M),$$
for all $n_1, \ldots, n_d>0$, where $I(M)$ is an invariant of $M$. Almost p-standard systems of parameters are generalization of standard systems of parameters for the modules which are not generalized Cohen-Macaulay. Similar to the standard systems of parameters, the almost p-standard property could be characterized by means of the notion of d-sequence.


\begin{proposition}\cite[Corollary 3.6]{NTCDTC1}\label{32}
Let $x_1, \ldots, x_d$ be a system of parameters of a finitely generated module $M$. The following statements are equivalent.
\begin{enumerate}[(a)]
\item $x_1, \ldots, x_d$ is an almost p-standard system of parameters;
\item $x_1^{n_1}, \ldots, x_i^{n_i}$ is a d-sequence on $M/(x_{i+1}^{n_{i+1}}, \ldots, x_d^{n_d})M$ for $i=1, \ldots, d$ and for all $n_1, \ldots, n_d>0$.
\end{enumerate}
\end{proposition}

Recall that a sequence $x_1, \ldots, x_s$ is a d-sequence on $M$ if $(x_1, \ldots, x_{i-1})M:x_ix_j=(x_1, \ldots, x_{i-1})M:x_j$ for all $1\leq i\leq j\leq s$. It is a strong d-sequence if $x_1^{n_1}, \ldots, x_s^{n_s}$ is a d-sequence on $M$ for any $n_1, \ldots, n_s>0$.

A local ring has an almost p-standard system of parameters if and only if it is a quotient of a Cohen-Macaulay local ring (see \cite{NTCDTC2}). Therefore most of local rings in commutative algebra have an almost p-standard system of parameters. 

In the next we have some technical results.

\begin{proposition}\label{33}
Let $x_1, \ldots, x_d$ be an almost p-standard system of parameters of a finitely generated $R$-module $M$. For $0\leq i<j\leq d$, denote $I=(x_{i+1}, \ldots, x_j)$. The sequence $x_1, \ldots, x_i, x_{j+1}^2, \ldots, x_d^2$ is an almost p-standard system of parameters of $M/I^nM$ for all $n>0$.
\end{proposition}
\begin{proof} Let $n_1, \ldots, n_d$ be positive integers. By \cite{NTCDTC1}, $x_{i+1}, \ldots, x_j$ is an almost p-standard system of parameters of $M_1=M/(x_1^{n_1}, \ldots, x_i^{n_i}, x_{j+1}^{n_{j+1}}, \ldots, x_d^{n_d})M$, in particular, it is a strong d-sequence on $M_1$. Due to \cite[Theorem 4.1]{NVT1}, we have
$$\ell(M_1/I^{n+1}M_1)=\sum_{t=0}^{j-i}f_{j-i-t}(I; M_1)\binom{n+t}{t},$$
where $f_{j-i}(I; M_1)=h^0(M_1)$ and for $t>0$,
\[\begin{aligned}
f_{j-i-t}(I; M_1)=&h^0(M/(x_1^{n_1}, \ldots, x_i^{n_i}, x_{i+1}, \ldots, x_{i+t+1}, x_{j+1}^{n_{j+1}}, \ldots, x_d^{n_d})M)\\
&-h^0(M/(x_1^{n_1}, \ldots, x_i^{n_i}, x_{i+1}, \ldots, x_{i+t}, x_{j+1}^{n_{j+1}}, \ldots, x_d^{n_d})M).
\end{aligned}\]
Denote $M_2=M/(x_{j+1}^{n_{j+1}}, \ldots, x_d^{n_d})M$. The length of the zero-th local cohomology module is computed in \cite[Corollary 4.4]{DTCNam} and we have
\[\begin{aligned}
h^0(M/(x_1^{n_1}, \ldots, &x_i^{n_i}, x_{i+1}, \ldots, x_{i+t}, x_{j+1}^{n_{j+1}}, \ldots, x_d^{n_d})M)\\
=&h^0(M_2/((x_1^{n_1}, \ldots, x_i^{n_i}, x_{i+1}, \ldots, x_{i+t})M_2)\\
=&\sum_{v=0}^{i}n_1\ldots n_ve(x_1, \ldots, x_v; (0:x_{v+1})_{M_2/(x_{v+2}, \ldots, x_{i+t})M_2})\\
&+n_1\ldots n_i\sum_{v=i+1}^{i+t}e(x_1, \ldots, x_v; (0:x_{v+1})_{M_2/(x_{v+2}, \ldots, x_{i+t})M_2}),
\end{aligned}\]
which does not depend on $n_{j+1}, \ldots, n_d\geq 2$ by \cite[Proposition 3.2]{DTCNam}.

On the other hand, using Lech's theorem and the definition of almost p-standard system of parameters \ref{31}, we have
\[\begin{aligned}
f_0(I; M_1)
=&\lim_{m \to \infty}\ \ \frac{1}{m^{j-i}}\ell(M_1/(x_{i+1}^m, \ldots, x_j^m)M_1)\\
=&\lim_{m \to \infty}\ \ \frac{1}{m^{j-i}}\ell(M/(x_1^{n_1}, \ldots, x_i^{n_i}, x_{i+1}^m, \ldots, x_j^m, x_{j+1}^{n_{j+1}}, \ldots, x_d^{n_d})M)\\
=&n_1\ldots n_i\sum_{v=j}^dn_{j+1}\ldots n_ve(x_1, \ldots, x_v; (0:x_{v+1})_{M/(x_{v+2}, \ldots, x_d)M}).
\end{aligned}\]
Replacing $n_{j+1}, \ldots, n_d$ by $2n_{j+1}, \ldots, 2n_d$, there are integers $\lambda_0, \ldots, \lambda_i$, $\lambda_{j+1}, \ldots, \lambda_d$ such that
\[\begin{aligned}
\ell\left(\frac{M/I^{n+1}M}{(x_1^{n_1}, \ldots, x_i^{n_i}, x_{j+1}^{n_{j+1}}, \ldots, x_d^{n_d})M/I^{n+1}M}\right)
=&\sum_{v=0}^in_1\ldots n_v\lambda_v\\
&+ n_1\ldots n_i\sum_{v=j+1}^dn_{j+1}\ldots n_v\lambda_v,\end{aligned}\]
for all $n_1, \ldots, n_d\geq 1$. Therefore, $x_1, \ldots, x_i, x_{j+1}^2, \ldots, x_d^2$ is an almost p-standard system of parameters of $M/I^{n+1}M$ for all $n\geq 0$.
\end{proof}

If $R$ is a generalized Cohen-Macaulay module then almost p-standard systems of parameters are the same as standard systems of parameters. Particularly they are unconditioned strong d-sequences. We obtain an immediate consequence of Proposition \ref{33} for generalized Cohen-Macaulay rings (compare \cite[Theorem 1.2]{LT}).

\begin{corollary}\label{34}
Let $R$ be a generalized Cohen-Macaulay local ring with a standard system of parameters $x_1, \ldots, x_d$. Let $n>0$. For some $0<i<d$, put $I=(x_1, \ldots, x_i)$. Then $R/I^n$ is a generalized Cohen-Macaulay ring and $x_{i+1}, \ldots, x_d$ is a standard system of parameters of $R/I^n$.

In particular, if $R$ is a Cohen-Macaulay local ring then $R/I^n$ is a Cohen-Macaulay ring and $h^0_I(n)=0$.
\end{corollary}

Before stating the main theorem of this section, we need one more lemma which is a key in the proof of the theorem.

\begin{lemma}\label{35}
Let $x_1,\ldots, x_d$ be an almost p-standard system of parameters of $R$. For $0\leq i<j\leq d$, denote $I=(x_{i+1}, \ldots, x_j)$. Let $t\in \{1,\ldots, i\}\cup\{j+1,\ldots, d\}$. Then
$$I^{n+1}:_Rx_t=I^n(I:_Rx_t)+0:_Rx_t.$$
\end{lemma}
\begin{proof}
The conclusion is proved by induction on $j-i$. The conclusion is obvious if $n=0$ or $j-i=0$. Suppose $j-i>0$. It suffices to show that
$$I^{n+1}:x_t\subseteq I(I^n:x_t)+0:_Rx_t,$$
for all $n>0$. We have
$$I^{n+1}=x_{i+1}^2I^{n-1}+x_{i+1}J^n+J^{n+1},$$
where $J=(x_{i+2}, \ldots, x_j)$. Taking an element $a\in I^{n+1}:x_t$, we write
$$x_ta=x_{i+1}^2a_1+x_{i+1}a_2+a_3,$$
where $a_1\in I^{n-1}$, $a_2\in J^n$ and $a_3\in J^{n+1}$. Then
$$a_1\in (x_t, J^n):x_{i+1}^2.$$
Note that by \cite[Proposition 3.4]{NTCDTC1}, $x_1, \ldots, x_{t-1}, x_{t+1}, \ldots, x_d$ is an almost p-standard system of parameters of $M/x_tM$. Combining with Proposition \ref{33}, the sequence $x_1, \ldots, x_i, x_{i+1}, x_{j+1}^2, \ldots, x_d^2$ (with $x_t$ being removed) is an almost p-standard system of parameters of $R/(x_t, J^n)$. Then
$$a_1\in (x_t, J^n):x_{i+1}^2=(x_t, J^n):x_{i+1}.$$
We are able to write $x_{i+1}a_1=x_tb+b_2$, where $b_2\in J^n$. In particular, $b\in (x_{i+1}I^{n-1}+J^n):x_t= I^n:x_t$, since $a_1\in I^{n-1}$. Hence
$$x_t(a-x_{i+1}b)=x_{i+1}a_2^\prime +a_3,$$
where $a_2^\prime=a_2+b_2\in J^n$. Note that $x_{i+1}b\in I(I^n:x_t)$. So we only need to show that $a-x_{i+1}b\in I(I^n:x_t)+0:_Rx_t$, or in other words, from beginning we might assume without lost of generality that $a_1=0$, that is
$$x_ta=x_{i+1}a_2+a_3.$$

We have $a_2\in [(x_t, J^{n+1}):x_{i+1}]\cap J^n$. The induction assumption applying to the ring $R/x_tR$ and the ideal $J=(x_{i+2}, \ldots, x_j)$ yields
$$(x_t, J^{n+1}):x_{i+1}=J^n((x_t, J):x_{i+1})+x_tR:x_{i+1}.$$
Hence
$$a_2\in [(x_t, J^{n+1}):x_{i+1}]\cap J^n=J^n((x_t, J):x_{i+1})+(x_tR:x_{i+1})\cap J^n.$$
Write
$$a_2=\sum_k\lambda_ku_k+u,$$
for $\lambda_k\in J^n$, $u_k\in (x_t, J):x_{i+1}$ and $u\in (x_tR:x_{i+1})\cap J^n$. We have $x_{i+1}u_k=x_tv_k+w_k$ for some $w_k\in J$ and $v_k\in (x_{i+1}, J):x_t=I:x_t$. Thus
$$x_{i+1}a_2=x_t\sum_k\lambda_kv_k+\sum_k\lambda_kw_k+x_{i+1}u.$$
Put $v=\sum_k\lambda_kv_k\in J^n(I:x_t)$ and $w=\sum_k\lambda_kw_k\in J^{n+1}$. We obtain
$$x_ta=x_tv+x_{i+1}u+w+a_3.$$

Now, it is worth noting that $x_1, \ldots, x_{t-1}, x_{t+1}, \ldots, x_d$ is an almost p-standard system of parameters of $R/x_tR$ due to Proposition \ref{32} and \cite[Proposition 3.4]{NTCDTC1}. We obtain $(0:_{R/x_tR}x_{i+1})\cap J^n(R/x_tR)=0$ as being shown in \cite[Corollary 3.7]{NTCDTC3}, or equivalently,
$$(x_tR:x_{i+1})\cap (x_t, J^n)=x_tR.$$
It shows that
$$u\in  (x_tR:x_{i+1})\cap J^n= x_tR\cap J^n.$$
Let $u=x_tu^\prime$, then $u^\prime\in J^n:x_t$. So $x_{i+1}u=x_t.x_{i+1}u^\prime$ where
$$x_{i+1}u^\prime\in x_{i+1}(J^n:x_t)\subseteq I(I^n:x_t).$$
To sum up, we have $x_ta=x_{i+1}a_2+a_3=x_tv+x_tx_{i+1}u^\prime+(w+a_3)$. Then
$$a-v-x_{i+1}u^\prime\in J^{n+1}:x_t=J(J^n:x_t)+0:_Rx_t.$$
The last equality holds by applying the induction assumption to $J=(x_{i+2}, \ldots, x_j)$. Therefore $a\in I(I^n:x_t)+0:_Rx_t$ and
$$I^{n+1}:x_t\subseteq I(I^n:x_t)+0:_Rx_t.$$
\end{proof}

\begin{corollary}\label{36}
Keep the notations and assumption as in Lemma \ref{35}. Let $t=1$ if $i\geq 1$ or $t=j+1$ if $i=0$. For each $n\geq 0$, we have
$$\bigcup_{s>0}I^{n+1}:\fm^s=I^{n+1}:_Rx_t=I^n(I:_Rx_t)+0:_Rx_t.$$
\end{corollary}
\begin{proof}
By Proposition \ref{33}, $x_1, \ldots, x_i, x_{j+1}^2, \ldots, x_d^2$ is an almost p-standard system of parameters of $R/I^{n+1}$. Particularly it is a strong d-sequence on $R/I^{n+1}$. Hence 

$$
\bigcup_{s>0}I^{n+1}:\fm^s=
\begin{cases}
I^{n+1}:_Rx_1&\mbox{ if } i\geq 1;\\
I^{n+1}:_Rx_{j+1}^2&\mbox{ if } i=0.
\end{cases}
$$
In the second case, following Lemma \ref{35}, we have 
\[\begin{aligned}
I^{n+1}:_Rx_{j+1}^2
&=I^n(I:_Rx_{j+1}^2)+0:_Rx_{j+1}^2\\
&=I^n(I:_Rx_{j+1})+0:_Rx_{j+1}\\
&=I^{n+1}:_Rx_{j+1}.
\end{aligned}\]
\end{proof}

Let $x_1, \ldots, x_d$ be an almost p-standard system of parameters of $R$ and $I=(x_{i+1}, \ldots, x_j)$ for $0\leq i<j\leq d$.  Let $t=1$ if $i\geq 1$ and $t=j+1$ if $i=0$. Denote $I_n=I^n:x_t$. We assume in addition that $R$ is {\it unmixed}, that is, $\dim R/\fp=\dim R$ for all associated prime ideals $\fp$ of $R$. In particular, $x_t$ is a regular element of $R$. Hence $I_1^2=I_2$ and $I$ is a reduction of $I_1$. Furthermore, from Lemma \ref{35} we get that $I_nI_m=I_{n+m}$ for all $n, m>0$. This observation together with Corollary \ref{36} show that for all $n\geq 0$, $I_n=I_1^n$ and
$$h^0_I(n)=\ell_R(I_1^{n+1}/I^{n+1}).$$
Therefore $h^0_I(n)$ is the Rees function of the pair $(I, I_1)$ in the sense of \cite{HPV}. Combining with the main theorem of Amao in \cite{AMAO} and Herzog-Puthenpurakal-Verma \cite[Corollary 4.7]{HPV}, we obtain the main theorem of this section.

\begin{theorem}\label{37}
Let $R$ be a quotient of a Cohen-Macaulay local ring. Suppose $R$ is unmixed. Let $x_1, \ldots, x_d$ be an almost p-standard system of parameters of $R$. For $0\leq i<j\leq d$, let $I$ be the ideal generated by $x_{i+1}, \ldots, x_j$. Then there is a polynomial $P_I(n)$ such that $h^0_I(n)=P_I(n)$ for all $n\gg 0$. Moreover, $\deg(P_I(n))+1$ is equal to the dimension of the graded module $\bigoplus_{n=1}^\infty I_1^n/I^n$ over the Rees algebra $\mathcal R(I)$.
\end{theorem}

\begin{remark}  Let $x_1, \ldots, x_d$ be an almost p-standard system of parameters of $R$. For $0\leq i<j\leq d$, let $I$ be the ideal generated by $x_{i+1}, \ldots, x_j$. Suppose $\mathrm{Ass} R \subseteq \mathrm{Assh}\,R \cup \{\frak m\}$, where $\mathrm{Assh}\,R = \{\frak p \in \mathrm{Ass}\,R \,:\, \dim R/\frak p =\dim R\}$ (e.g. $R$ is unmixed). Then we have $0: x_t = H^0_{\frak m}(R)$ for all $t \ge 1$. Set $R_1 = R/H^0_{\frak m}(R)$ and $I_1 = IR_1$. Notice that the image of $x_1, \ldots, x_d$ in $R_1$ is an almost p-standard system of parameters of $R_1$. Since $R_1$ is unmixed and $h^0_{I,R}(n) = h^0_{I_1,R_1}(n) + \ell(H^0_{\frak m}(R))$
for all $n \ge 1$, we have $h^0_{I,R}(n)$ is of polynomial type. We conjecture that $h^0_{I,R}(n)$ is always of polynomial type for general rings. However the proof seems to be complicated.
\end{remark}

As we have discussed before, typical examples of almost p-standard systems of parameters are standard systems of parameters. The theorem particularly asserts that if $I$ is an ideal in a generalized Cohen-Macaulay local ring which is generated by part of a standard system of parameters, then the function $h^0_I(n)$ is of polynomial type. We even have more, the polynomial can be computed in terms of length of certain local cohomology modules of the ring. This is done in the next part of this section. For the results on generalized Cohen-Macaulay local rings, we refer to \cite{NVT2}.

From Lemma 3.5, we obtain the following well known property of standard system of parameters.

\begin{corollary}[\cite{NVT2}, Corollary 2.6 (v)] \label{38}
Let $R$ be a generalized Cohen-Macaulay local ring. Let $x_1, \ldots, x_d$ be a standard system of parameters of $R$. Denote $I=(x_1, \ldots, x_i)$ for some $0\leq i\leq d$. Then $$I^{n+1}:_Rx_1=I^n+0:_Rx_1,$$ for all $n\geq 0$.
\end{corollary}
\begin{proof} It suffices to show that $I^{n+1}:_Rx_1\subseteq I^n + 0:_R x_1$.

We have $I^{n+1}=x_1I^n+J^{n+1}$ where $J=(x_2, \ldots, x_i)$. Let $a\in R$ such that $ax_1\in I^{n+1}$, hence $ax_1=x_1b+c$ for some $b\in I^n$, $c\in J^{n+1}$. We have $a-b\in J^{n+1}:x_1$. Lemma \ref{35} applies and we have
$$a-b\in J^n+0:_Rx_1.$$
So $a\in I^n+0:_Rx_1$.
\end{proof}
The case $i = d$ of the following is a well-known result (see \cite[Theorem 4.1]{NVT2})
\begin{theorem}\label{39}
Let $R$ be a generalized Cohen-Macaulay local ring and let $I$ be an ideal of $R$ generated by $i$ elements, $i < d$, in a standard system of parameters of $R$. Then
$$h_I^0(n)=h^0(R)+\sum_{t=0}^{i-1}\Big(\sum_{j=0}^t\binom{t}jh^{j+1}(R)\Big)\binom{n+t}{t},$$
where $h^j(R)$ is the length of the $j$-th local cohomology module $H^j_\fm(R)$. In particular, $h_I^0(n)=0$ if $i<\mathrm{depth}(R)$ and $\deg(h_I^0(n))=i-1$ if $\mathrm{depth}(R) \le i < d$.
\end{theorem}
\begin{proof}
Let $x_1, \ldots, x_d$ be a standard system of parameters of $R$ so that $x_1, \ldots, x_i$ generate $I$. Denote $R_1=R/x_1R$ and $I_1=(x_2, \ldots, x_i)R_1=IR_1$. The ring $R_1$ is again a generalized Cohen-Macaulay module. Recall that we denote $h^j(R)=\ell_R(H^j_\fm(R))$ for $j\not=\dim R$ and $h^0_{I, R}(n):=h^0_I(n)=\ell(H^0_\fm(R/I^{n+1}))$. We first prove the following claim which is necessary for induction process later.
\medskip

\noindent {\bf Claim.} $h_{I, R}^0(n)-h_{I, R}^0(n-1)=h_{I_1, R_1}^0(n)-h^0(R)-h^1(R),$ for all $n\geq 1$.
\medskip

In oder to prove the claim, we note by using Corollary \ref{38} that the map $R/I^{n+1}\stackrel{x_1}{\longrightarrow}R/I^{n+1}$ induces the short exact sequence
$$0\longrightarrow R/(I^{n}+0:x_1) \stackrel{x_1}{\longrightarrow}R/I^{n+1}\longrightarrow R_1/I_1^{n+1}\longrightarrow 0.$$
Then the commutative diagram
\begin{displaymath}
\xymatrix{
0\ar[r]&R/(I^n+0:x_1)  \ar[r]^{x_1} \ar[d]_{x_{i+1}} &
R/I^{n+1}\ar[d]^{x_{i+1}}\ar[r]&R_1/I_1^{n+1}\ar[d]^{x_{i+1}}\ar[r]&0\\ 0\ar[r]& R/(I^n+0:x_1) \ar[r]^{x_1} & R/I^{n+1}\ar[r]&R_1/I_1^{n+1}\ar[r]&0}
\end{displaymath}
gives rise to an exact sequence
\begin{multline*}
0\longrightarrow  ((I^n+0:x_1):x_{i+1})/(I^n+0:x_1)\longrightarrow  (I^{n+1}:x_{i+1})/I^{n+1}\\
\longrightarrow (I_1^{n+1}R_1:x_{i+1})/ I_1^{n+1}
\longrightarrow \Ker(f)\longrightarrow  0,
\end{multline*}
where $f$ is the multiplication map
$$R/(x_{i+1}R+ I^n+0:x_1)\stackrel{x_1}{\longrightarrow}R/(x_{i+1}R+I^{n+1}).$$
We have
$$\ell((I^{n+1}:x_{i+1})/I^{n+1})=\ell(H^0_\fm(R/I^{n+1})),$$
$$\ell((I_1^{n+1}R_1:x_{i+1})/I_1^{n+1})=\ell(H^0_\fm(R_1/I_1^{n+1})),$$
\[\begin{aligned}
\ell(((I^n+0:x_1):x_{i+1})/(I^n+0:x_1))
&=\ell(H^0_\fm(R/I^n+0:x_1))\\
&=\ell(H^0_\fm(R/I^n))-\ell(H^0_\fm(R)),
\end{aligned}\]
where $H^0_\fm(R)=0:_Rx_1$. Moreover, $(x_{i+1}R+I^{n+1}):x_1=I^n+x_{i+1}R:x_1$ by applying Corollary \ref{38} to the ring $R/x_{i+1}R$ and the sequence $x_1, \ldots, x_i$. Hence
\[\begin{aligned}
\Ker(f)&=((x_{i+1}R+I^{n+1}):x_1)/(x_{i+1}R+I^n+0:_Rx_1)\\
&\simeq (I^n+x_{i+1}R:x_1)/(I^n+x_{i+1}R+0:_Rx_1)\\
&\simeq x_{i+1}R:x_1/(x_{i+1}R+0:_Rx_1)\\
&\simeq H^0_\fm(R/(x_{i+1}R+0:_Rx_1)).
\end{aligned}\]
It implies that $\ell(\Ker(f))=\ell(H^0_\fm(R/x_{i+1}R))-\ell(H^0_\fm(R))=\ell(H^1_\fm(R))$. So the claim is proved.
\medskip

The theorem is now proved by induction on $i$. If $i=1$, we have from the claim $h_I^0(n)-h_I^0(n-1)=h^0(R/x_1R)-h^1(R)-h^0(R)=0$, hence
$$h_I^0(n)=h^0_I(n-1)=h^0_I(0)=h^0(R)+h^1(R).$$
Assume $i>1$. Combining the induction assumption with the claim, we have
\[\begin{aligned}
h_I^0(n)-h_I^0(n-1)=&h^0(R/x_1R)-h^0(R)-h^1(R)\\
&+\sum_{t=0}^{i-2}\Big(\sum_{j=0}^t\binom{t}jh^{j+1}(R/x_1R)\Big)\binom{n+t}{t}.\end{aligned}\]
For the generalized Cohen-Macaulay ring $R$, we have $h^j(R/x_1R)=h^j(R)+h^{j+1}(R)$. A simple computation then gives us
$$h_I^0(n)=h^0(R)+\sum_{t=0}^{i-1}\Big(\sum_{j=0}^t\binom{t}jh^{j+1}(R)\Big)\binom{n+t}{t}.$$
\end{proof}

In a more general case when $I$ is generated by part of a system of parameters which is not almost p-standards, we probably expect that the assertion of Theorem 3.7 still holds true for I. Surprisingly, it is not the case. We end with the following example.

\begin{example}\label{310}
By Cutkosky, Ha, Srinivasan and Theodorescu \cite[Theorem 2.2]{CTST}, there exist a regular local ring $(A,\frak n)$ and an ideal $I$ such that $h^0_I(n)$ is not asymptotic to a polynomial. We choose a set of generators of $I$, say $a_1, ..., a_r$. Let $S = A[T_1, ..., T_r]_{(T_1,..., T_r)+\frak n}$, we have $\dim S = \dim A + r$. Let $J = (T_1+a_1, ..., T_r+a_r)$. It is easy to check that $S/J \cong A$. So $J$ is generated by part of a system of parameters of $S$. Since $S$ is regular, we have  $S/J^n$ is Cohen-Macaulay for all $n$.

Consider $A$ as an $S$-module we have $\mathrm{Ann}_SA = (T_1, ..., T_r)$, so $A/J^nA = A/I^n$. Let $\frak n^\prime = (T_1,..., T_r)+\frak nS$, then $H^0_{\frak n^\prime}(S/J^n) = 0$ and $H^0_{\frak n^\prime}(A/J^nA) \cong H^0_{\frak n}(A/I^n)$.

Take the idealization $R = S \ltimes A$ and denote $\frak q = ((T_1+a_1, 0), \ldots, (T_r+a_r, 0))$. The ideal $\frak q$ in the local ring $R$ is generated by part of a system of parameters and 
$$R/\frak q^n \simeq S/J^n \ltimes A/I^n,$$
for all $n$. Let $\fm$ be the maximal ideal of $R$ we have
$$\ell(H^0_{\fm}(R/\frak q^n)) = \ell(H^0_{\frak n}(A/I^n)),$$
which is not asymptotic to a polynomial in $n$.

Finally, it should be remarked that the ring $R$ is not a domain. It would be interesting to find a counterexample to be a domain.
\end{example}


\medskip
\noindent{\bf Acknowledgments.} The authors would like to thank the anonymous referees for their valuable comments and suggestions to improve the presentation of the paper. The first and the second authors are funded by Vietnam National Foundation for Science and Technology Development (NAFOSTED) under grant number 101.04-2015.26. This paper was written while the third author was visiting Vietnam Institute for Advanced Study in Mathematics. He would like to thank the VIASM for hospitality and financial support.


\end{document}